\documentclass[12pt]{amsart}
\usepackage[utf8]{inputenc}

\usepackage{hyperref}
\hypersetup{
    colorlinks=true,
    linkcolor=cyan,
    citecolor=red,
}

\usepackage{geometry}
\usepackage{setspace}
\usepackage{amsmath}
\usepackage{amsfonts}
\usepackage{amssymb}

\usepackage{xpatch}
\usepackage{amsthm}
\newtheorem{theorem}{Theorem}[section]

\newtheorem{lemma}[theorem]{Lemma}

\theoremstyle{remark}
\newtheorem{remark}[theorem]{Remark}

\def\R{\mathbb{R}}
\def\N{\mathbb{N}}
\def\d{\mathrm{d}}
\def\p{\partial}
\def\l{\left(}
\def\r{\right)}

\DeclareMathOperator*{\esssup}{ess\,sup}
\DeclareMathOperator*{\B}{B}

\allowdisplaybreaks[1]

\begin{document}

\title{Gagliardo-Nirenberg inequality with H\"older norms}

\author{Mengxia Dong}
\address{Division of General Education, Shenzhen University of Advanced Technology, Shenzhen 518083, Guangdong, China}
\email{dongmengxia@suat-sz.edu.cn}




\begin{abstract}
    The classical Gagliardo-Nirenberg inequality, a fundamental result in interpolation theory, traditionally involves Lebesgue norms of functions and their derivatives. We extend this inequality by introducing an interpolation lemma that bridges Lebesgue and Hölder spaces, replacing conventional Sobolev norms with appropriate Hölder norms. This extension not only broadens the range of admissible parameters but also establishes a unified framework for inequalities in Sobolev and Hölder spaces, offering new insights and applications.
\end{abstract}

\maketitle

\tableofcontents

\section{Introduction}

For any weakly differentiable function $u\in W^{1,p}$ with $1\le p<n$, the classical Sobolev embedding theorem states that
\begin{equation}\label{SobolevInequality}
    \|u\|_{p^\ast}\le C\|\nabla u\|_p,
\end{equation}
where $p^\ast=\frac{np}{n-p}$ and $C>0$ is independent of $u$. If $p>n$, then according to the Morrey's inequality, for the continuous representative, the following holds:
\begin{equation}\label{MorreyInequality}
    \|u\|_{C^{0,1-\frac{n}{p}}}\le C\|\nabla u\|_p.
\end{equation}
These results are well documented in standard references such as \cite{brezis2011functional} or \cite{evans2010partial}.

As fundamental tools in  Functional Analysis, Partial Differential Equations, Sobolev spaces describe the regularity and integrability of functions, while Sobolev embeddings reveal the connection between these two aspects. However, as a space for continuous functions, H\"older spaces mainly emphasize the regularity of functions, but lack a description of their integrability. Both regularity and integrability are essential for describing the rate of change of functions. Based on this commonality, L. Nirenberg introduced a formal integrability parameter $p<0$ for Hölder spaces \cite{nirenberg2011elliptic}, successfully unifying these two types of spaces. Here is the definition:
\begin{equation*}
    |u|_p=[u]_{C^{p_1,p_2}}
    =\sum_{|\alpha|=p_1}\sup_{x\ne y}\frac{|D^\alpha u(x)-D^\alpha u(y)|}{|x-y|^{p_2}},
\end{equation*}
where
\begin{equation}\label{Notation}
    \begin{aligned}
        p_1&=\left[-\frac{n}{p}\right], \\
        p_2&=-\frac{n}{p}-p_1.
    \end{aligned}
\end{equation}
Under this notation, when $p>n$, it is easy for us to verify that
\begin{align*}
    p^\ast_1&=\left[-\frac{n}{p^\ast}\right]=\left[1-\frac{n}{p}\right]=0, \\
    p^\ast_2&=-\frac{n}{p^\ast}-p^\ast_1=1-\frac{n}{p}.
\end{align*}
Thus
\begin{equation}\label{MorreyInequalitySeminorm}
    |u|_{p^\ast}=[u]_{C^{0,1-\frac{n}{p}}}\le C\|\nabla u\|_p.
\end{equation}
It can be seen that (\ref{MorreyInequalitySeminorm}), as a Morrey's inequality with only seminorms, closely resembles the Sobolev inequality (\ref{SobolevInequality}). This similarity hints at a deeper connection between the two. The notation used here unifies the parameters of Sobolev and Hölder spaces, suggesting a framework where these spaces can be studied together.

From this perspective, Sobolev and Hölder spaces, though defined differently, can be seen as part of a unified, larger space. This raises the question of whether functional inequalities valid in Sobolev spaces can extend to this broader setting. The main contribution of this paper is to extend the Gagliardo-Nirenberg inequality from Sobolev spaces to a new space encompassing both Sobolev and Hölder spaces.

Next, we review some relevant results of the Gagliardo-Nirenberg inequality. In L. Nirenberg's article \cite{nirenberg2011elliptic}, he presents the famous interpolation inequality.
\begin{equation}\label{GagliardoNirenbergInequality}
    \|D^l u\|_q \le C \|D^k u\|_p^\theta \|u\|_r^{1-\theta},
\end{equation}
for $p,q,r\ge 1$ and all functions $u\in C^\infty_0(\R^n)$ with a constant $C$ independent of $u$. Here, $k,l\in\mathbb{N}$ with $k>l$, and the parameters satisfy
\begin{equation}\label{ParameterCondition0}
    \frac{1}{q}-\frac{l}{n} = \theta\cdot\l\frac{1}{p}-\frac{k}{n}\r+(1-\theta)\cdot\frac{1}{r},
\end{equation}
for all $\theta$ in the interval
\begin{equation*}
    \frac{l}{k}\le\theta\le 1.
\end{equation*}
This inequality was also proven by E. Gagliardo in \cite{gagliardo1959ulteriori} and independently by L. Nirenberg, therefore, the inequality is referred to as the Gagliardo-Nirenberg inequality. Gagliardo-Nirenberg inequality has diverse applications in Functional Analysis, Partial Differential Equations, and Mathematical Physics. It is particularly useful in studying the regularity properties of solutions to elliptic and parabolic equations, interpolation theory, and in establishing compactness results for function spaces.

The traditional Gagliardo-Nirenberg inequality applies only to Sobolev spaces. In 1995, A. Kufner and A. Wannebo \cite{kufner1995interpolation} modified the Gagliardo-Nirenberg inequality (\ref{GagliardoNirenbergInequality}) by replacing the $L^r$-norm of $u$, $\|u\|_r$ with the H\"older seminorm $|u|_r$ on the right-hand side. The specific results are as follows:

Let $p,q\ge 1$, $-\infty<r<-n$ and $\theta\in(0,1)$ satisfy the same parameter conditions as in (\ref{ParameterCondition0}). Then for all functions $u\in C^\infty_0(\R^n)$, there exists a constant $C$ independent of $u$ such that
\begin{equation}\label{GagliardoNirenbergInequalityOneHolder}
    \|D^l u\|_q \le C \|D^k u\|_p^\theta |u|_r^{1-\theta},
\end{equation}

Subsequently, in 2018, A. Molchanova, T. Roskovec, and F. Soudsk\'y  extended the inequality further by replacing both Lebesgue norms, $\|u\|_q$ and $\|u\|_r$, with H\"older seminorms $|u|_q$ and $|u|_r$ \cite{soudsky2018interpolation}. However, they only considered the basic case $k=l=0$ and did not extend it to higher-order derivatives. The specific results are as follows:

Let $p\in[1,+\infty]$, $q,r\in(-\infty,-n)$ and $\theta\in(0,1)$ satisfy the parameter conditions
\begin{equation}\label{ParameterCondition02}
    \frac{1}{q}=\frac{\theta}{p}+\frac{1-\theta}{r}.
\end{equation}
Then for all functions $u\in C^\infty_0(\R^n)$, there exists a constant $C$ independent of $u$ such that
\begin{equation}\label{GagliardoNirenbergInequalityTwoHolder}
    |u|_q \le C \|u\|_p^\theta |u|_r^{1-\theta},
\end{equation}
It is not hard to see that the parameter condition (\ref{ParameterCondition02}) is a special case of parameter condition (\ref{ParameterCondition0}) of the traditional Gagliardo-Nirenberg inequality (\ref{GagliardoNirenbergInequality}).

In this paper, we establish an extended version of the classical Gagliardo-Nirenberg inequality (\ref{GagliardoNirenbergInequality}), in which Lebesgue norms are replaced by Hölder norms under suitable parameter conditions. Unlike previous results that primarily focused on seminorms, our work generalizes the inequality to full norms, significantly broadening its applicability. Moreover, we extend the parameter range to a larger interval, encompassing all cases covered by (\ref{GagliardoNirenbergInequality}), (\ref{GagliardoNirenbergInequalityOneHolder}), and (\ref{GagliardoNirenbergInequalityTwoHolder}), while also introducing new scenarios not previously addressed. This advancement not only unifies existing results but also provides a more comprehensive framework for analyzing functional inequalities in Sobolev and Hölder spaces. 

For convenience in exposition and proof, we slightly adjust the notation given by L. Nirenberg in \cite{nirenberg2011elliptic}. Now let us introduce our new notation. Define $X^p$ as a space that encompasses both Lebesgue and H\"older spaces:
\begin{align*}
    0<p<\infty:& \quad X^p=L^p, \\
    p=\infty:& \quad X^p=L^\infty, \\
    -\infty<p<0:& \quad X^p=C^{p_1,p_2},
\end{align*}
where $p_1=-\left[\frac{n}{p}+1\right]$, $p_2=-\frac{n}{p}-p_1$ when $p<0$. Define the norm of $X^p$ as:
\begin{align*}
    0<p<\infty:& & \|u\|_{X^p(\R^n)}&=\|u\|_{L^p(\R^n)}=\l\int_{\R^n}|u|^p\d x\r^\frac{1}{p}, \\
    p=\infty:& & \|u\|_{X^p(\R^n)}&=\|u\|_{L^\infty(\R^n)}=\esssup_{x\in\R^n}|u(x)|, \\
    -\infty<p<0:& & \|u\|_{X^p(\R^n)}&=\|u\|_{C^{p_1,p_2}(\R^n)}
    =\|u\|_{C^{p_1}(\R^n)}+\left[D^{p_1} u\right]_{C^{0,p_2}(\R^n)} \\
    & & &=\max_{|\alpha|\le p_1}\sup_{x\in\R^n}|D^\alpha u|
    +\sum_{|\alpha|=p_1}\sup_{\substack{x,y\in\R^n \\ x\ne y}}\frac{|D^\alpha u(x)-D^\alpha u(y)|}{|x-y|^{p_2}}.
\end{align*}
This notation connects Hölder and Lebesgue spaces. Although functions in Hölder spaces are not necessarily integrable, assigning them a negative integrability parameter makes it compatible with the integrability parameter in Lebesgue spaces. The idea of this notation shares many similarities with the Morrey-Campanato spaces. The literature related to Morrey-Campanato spaces has been studied by various authors S. Campanato \cite{campanato1963proprieta}, F. John and L. Nirenberg \cite{john1961functions}, N.G. Meyers \cite{meyers1964mean}, where G. Stam proved the Interpolation Theorem for this space in \cite{stampacchia1965spaces}. However, the Morrey-Campanato space approaches the relationship between Sobolev and Hölder spaces through the lens of mean oscillation, which emphasizes regularity. In contrast, both our work and Nirenberg’s notation share the same perspective, focusing on extending integrability to unify these spaces. While Nirenberg achieves this by generalizing integrability conditions, our approach further extends this framework by broadening the parameter ranges and generalizing functional inequalities. This shared perspective not only aligns with Nirenberg’s framework but also strengthens the understanding of Sobolev and Hölder spaces as part of a unified structure, distinct from the regularity-based viewpoint of Morrey-Campanato spaces.

Similar to the Sobolev space, fix $-\infty<\frac{1}{p}<+\infty$ and let $k$ be a non-negative integer, if for each multi-index $\alpha$ with $\alpha\le k$,
\begin{equation*}
    D^\alpha u\in X^p(\R^n),
\end{equation*}
then we say
\begin{equation*}
    u\in X^{k,p}(\R^n).
\end{equation*}

Here is the main theorem of this paper.
\begin{theorem}\label{TheoremGagliardoNirenbergInequalityGeneral}
    Let $p,q,r\in(-\infty,0)\cup[1,+\infty)$, $k,l\in\mathbb{N}^+$, $k>l$ and $\frac{n}{p}\notin\{1,\cdots,k-l\}$, then for all $u\in C_0^\infty(\R^n)$ with $u\in X^{k,p}(\R^n)\cap X^r(\R^n)$, we have $u\in X^{l,q}(\R^n)$ and there exists a constant $C$ independent of $u$ such that
    \begin{equation}\label{GagliardoNirenbergInequalityGeneral}
        \|D^l u\|_{X^q(\R^n)} \le C \|D^k u\|_{X^p(\R^n)}^\theta \|u\|_{X^r(\R^n)}^{1-\theta},
    \end{equation}
    where
    \begin{equation*}
        \frac{1}{q}-\frac{l}{n} = \theta\cdot\l\frac{1}{p}-\frac{k}{n}\r+(1-\theta)\cdot\frac{1}{r},
    \end{equation*}
    for all $\theta$ in the interval
    \begin{equation*}
        \frac{l}{k}\le\theta\le 1.
    \end{equation*}
\end{theorem}

This paper is organized as follows: In section 2, we will use the new notation to present a unified form of the Sobolev inequality and the Morrey's inequality. The main purpose is to facilitate subsequent calculations. In Section 3, we established a crucial interpolation lemma related to both Lebesgue and H\"older spaces. This lemma plays a key role in problem-solving, and importantly, it shows that, provided the parameter conditions are maintained, the interpolation inequality still holds even when extended to H\"older spaces. In Section 4, we complete our conclusions using mathematical induction and the interpolation lemma.

\section{Sobolev Inequalities}

In the previous section, we introduced a new notation that unifies the Sobolev inequality and Morrey’s inequality into a single framework. This unified inequality extends the parameter $p$ to all negative values in the Hölder space, providing a broader perspective on their relationship. It will serve as a key tool in the proofs developed in later sections. Let us now verify this result:
\begin{theorem}\label{TheoremSobolevGeneral}
    Let $p\in(-\infty,0)\cup[1,n)\cup(n,+\infty)$, then for all $u\in C_0^\infty(\R^n)$ with $u\in X^{1,p}(\R^n)$, we have $u\in X^{p^\ast}(\R^n)$ and there exists a constant $C$ independent of $u$ such that
    \begin{equation}\label{SobolevInequalityGeneral}
        \|u\|_{X^{p^\ast}(\R^n)}\le C\|D u\|_{X^p(\R^n)},
    \end{equation}
    where
    \begin{equation*}
        \frac{1}{p^\ast}=\frac{1}{p}-\frac{1}{n}.
    \end{equation*}
\end{theorem}

\begin{proof}
    When $p\in[1,n)$, both $p$ and $p^\ast$ are positive, the inequality (\ref{SobolevInequalityGeneral}) is the standard Sobolev inequality
    \begin{equation*}
        \|u\|_{p^\ast}\le C\|D u\|_p.
    \end{equation*}
    When $p\in(n,+\infty)$, $p$ is positive and $p^\ast$ is negative, direct calculation show us
    \begin{align*}
        p^\ast_1&=-\left[\frac{n}{p^\ast}+1\right]=0, \\
        p^\ast_2&=-\frac{n}{p^\ast}-p^\ast_1=1-\frac{n}{p}.
    \end{align*}
    In this case the inequality (\ref{SobolevInequalityGeneral}) becomes the Morrey's inequality:
    \begin{equation*}
        \|u\|_{C^{0,1-\frac{n}{p}}}\le C\|D u\|_p.
    \end{equation*}
    When $p\in(-\infty,0)$, $p^\ast<0$ is negative. We find
    \begin{align*}
        p^\ast_1&=-\left[\frac{n}{p^\ast}+1\right]
        =-\left[\frac{n}{p}\right]
        =p_1+1, \\
        p^\ast_2&=-\frac{n}{p^\ast}-p^\ast_1
        =1-\frac{n}{p}-(p_1+1)
        =-\frac{n}{p}-p_1
        =p_2.
    \end{align*}
    Therefore, when $p$ is negative, we obtain the H\"older equality 
    \begin{equation}\label{HolderEquality}
        \begin{aligned}
            \|u\|_{X^{p^\ast}(\R^n)}&=\|u\|_{C^{p^\ast_1}(\R^n)}
            +\left[D^{p_1^\ast} u\right]_{C^{0,p_2^\ast}(\R^n)} \\
            &=\|Du\|_{C^{p_1}(\R^n)}+\left[D^{p_1}(Du)\right]_{C^{0,p_2}(\R^n)}
            =\|D u\|_{X^p(\R^n)}.
        \end{aligned}
    \end{equation}    
    Which completes the proof.
\end{proof}

\begin{remark}
    We observe that when $p<0$, the Sobolev inequality transforms into an equality under these conditions, indicating the equivalence of the spaces $X^{p^\ast}(\R^n)$ and $X^{1,p}(\R^n)$ in this framework. However, in the theorem, we do not separate this case but instead represent it uniformly as an inequality. The reason is to maintain consistency for subsequent calculations, and because our main focus is on cases that span both H\"older and Lebesgue spaces.
\end{remark}

\begin{remark}
    If we suppose there is a space $Y$ with the norm $\|\cdot\|_Y$ and assume
    \begin{equation*}
        \|u\|_{X^\infty}:=\|u\|_Y.
    \end{equation*}
    We would like to apply it on the unified format notation in order to extend the range of $p$ in Theorem \ref{TheoremSobolevGeneral}. Thus we must have:
    \begin{gather*}
        \|u\|_Y=\|u\|_{X^\infty}\le C\|Du\|_{X^n}, \\
        \|Du\|_{L^\infty}\le\|u\|_{X^{-n}}\le C\|Du\|_{X^\infty}=C\|Du\|_Y.
    \end{gather*}
    These two inequalities show that
    \begin{equation*}
        W^{1,n}\hookrightarrow Y\hookrightarrow L^\infty.
    \end{equation*}
    However, this is impossible. Therefore, we cannot provide a perfect notation for borderline cases.
\end{remark}

By iteratively applying Theorem \ref{TheoremSobolevGeneral}, we can express the generic Sobolev inequality involving derivatives of any order more compactly.

\begin{theorem}\label{TheoremSobolevArbitraryOrder}
    Let $p\in(-\infty,0)\cup[1,+\infty)$, $k,l\in\N^+$, $k>l$ and $\frac{n}{p}\notin\{1,\cdots,k-l\}$. Then for all $u\in C_0^\infty(\R^n)$ with $u\in X^{k,p}(\R^n)$, we have $u\in X^{l,q}(\R^n)$ and there exists a constant $C$ independent of $u$ such that
    \begin{equation}\label{SobolevInequalityArbitraryOrder}
        \|D^l u\|_{X^q(\R^n)}\le C\|D^k u\|_{X^p(\R^n)},
    \end{equation}
    where
    \begin{equation*}
        \frac{1}{q}-\frac{l}{n}=\frac{1}{p}-\frac{k}{n}.
    \end{equation*}
\end{theorem}

\begin{proof}
    Assume $p_0=p$, let $p_{i+1}$ be the conjugate Sobolev index of $p_i$, i.e.
    \begin{equation*}
        \frac{1}{p_{i+1}}=\frac{1}{p_i}-\frac{1}{n} \quad  (i=0,1,\cdots,k-l-1).
    \end{equation*}
    Thus we have
    \begin{equation*}
        \frac{1}{p_{k-l}}=\frac{1}{p_{k-l-1}}-\frac{1}{n}
        =\cdots=\frac{1}{p}-\frac{k-l}{n}=\frac{1}{q}.
    \end{equation*}
    Since $\frac{n}{p}\notin\{1,\cdots,k-l\}$, consequently $p_i\ne n$, then repeatedly applying (\ref{SobolevInequalityGeneral}) we obtain
    \begin{equation*}
        \|D^l u\|_{X^q(\R^n)}
        \lesssim\|D^{l+1} u\|_{X^{p_{k-l-1}}(\R^n)}
        \lesssim\cdots
        \lesssim\|D^k u\|_{X^p(\R^n)}.
    \end{equation*}
\end{proof}

\section{Interpolation Lemma}

In \cite{nirenberg2011elliptic}, L. Nirenberg mentioned an interpolation lemma but did not provide its complete proof. In this section, we will establish a crucial interpolation inequality that spans Lebesgue and Hölder spaces, filling this gap with a foundational yet detailed proof. This inequality demonstrates that interpolation can be continuously applied between these spaces, effectively bridging their connection. The result is not only of independent interest but also plays a vital role in the proof of our main theorem.

\begin{lemma}\label{LemmaInterpolation}
    Assume $-\infty<\lambda<\mu<\nu<\infty$ and $u\in C_0^\infty(\R^n)$, then
    \begin{equation}\label{interpolationInequality}
        \|u\|_{X^\frac{1}{\mu}(\R^n)}\le C\|u\|_{X^\frac{1}{\lambda}(\R^n)}^{\eta}\|u\|_{X^\frac{1}{\nu}(\R^n)}^{1-\eta}
    \end{equation}
    where
    \begin{equation*}
        \mu=\eta\lambda+(1-\eta)\nu,
    \end{equation*}
    and $C$ is independent of $u$.
\end{lemma}

\begin{remark}
     Some special cases of this lemma have been studied. In particular
     \begin{enumerate}
         \item When $\lambda\ge0$, all three norms are Lebesgue norms, making inequality (\ref{interpolationInequality}) essentially the interpolation inequality for Lebesgue spaces.
         \item When $\nu\le0$, all three norms are H\"older norms, you can refer to the relevant studies in Appendix 1 of Chapter 6 in D. Gilbarg, and N. S. Trudinger's authoritative book \cite{gilbarg1977elliptic}.
     \end{enumerate}
\end{remark}

Before proving the interpolation lemma, let’s first examine two lemmas. One of them illustrate the transitivity of interpolation, which is a straightforward application of the Reiteration Theorem found by A. Calder\'on in \cite{calderon1964intermediate}. Readers may also refer the theorem in H. Triebel's book \cite{triebel1995interpolation} and A. Lunardi's books \cite{lunardi2012analytic} and \cite{lunardi2018interpolation}; it presents a more general form within Interpolation Theory. This lemma allows us to focus on proving key parameter nodes, while the general case can be directly deduced from it, we will frequently employ this method in subsequent proofs.

\begin{lemma}\label{LemmaTransitivityOfInterpolation}
    For $-\infty<\mu_0<\mu_1<\mu_2<\mu_3<\infty$, $i=1,2$, assume we have the interpolation inequalities
    \begin{equation*}
        \|u\|_{X^\frac{1}{\mu_i}(\R^n)}
        \lesssim\|u\|_{X^\frac{1}{\mu_{i-1}}(\R^n)}^{\eta_i}
        \|u\|_{X^\frac{1}{\mu_{i+1}}(\R^n)}^{1-\eta_i},
    \end{equation*}
    where
    \begin{equation*}
        \mu_i=\eta_i\mu_{i-1}+(1-\eta_i)\mu_{i+1}.
    \end{equation*}
    Then for $j=1,2$, we have the interpolation inequalities
    \begin{equation*}
        \|u\|_{X^\frac{1}{\mu_j}(\R^n)}
        \lesssim \|u\|_{X^\frac{1}{\mu_0}(\R^n)}^{\theta_j}
        \|u\|_{X^\frac{1}{\mu_3}(\R^n)}^{1-\theta_j},
    \end{equation*}
    where
    \begin{equation*}
        \mu_j=\theta_j\mu_0+(1-\theta_j)\mu_3.
    \end{equation*}
\end{lemma}

\begin{proof}
    We only prove for $j=1$, the proof of the other case is similar. From the condition we have 
    \begin{align*}
        \mu_1&=\eta_1\mu_0+(1-\eta_1)\mu_2 \\
        &=\eta_1\mu_0+(1-\eta_1)(\eta_2\mu_1+(1-\eta_2)\mu_3) \\
        &=\eta_1\mu_0+(1-\eta_1)\eta_2\mu_1+(1-\eta_1)(1-\eta_2)\mu_3.
    \end{align*}
    Therefore
    \begin{equation*}
        \mu_1=\frac{\eta_1}{1-\eta_2+\eta_1\eta_2}\mu_0
        +\l 1-\frac{\eta_1}{1-\eta_2+\eta_1\eta_2}\r\mu_3.
    \end{equation*}
    Choose
    \begin{equation*}
        \theta_1=\frac{\eta_1}{1-\eta_2+\eta_1\eta_2}.
    \end{equation*}
    Then
    \begin{equation*}
        \|u\|_{X^\frac{1}{\mu_1}(\R^n)}
        \lesssim\|u\|_{X^\frac{1}{\mu_0}(\R^n)}^{\eta_1}
        \|u\|_{X^\frac{1}{\mu_2}(\R^n)}^{1-\eta_1}
        \lesssim\|u\|_{X^\frac{1}{\mu_0}(\R^n)}^{\eta_1}
        \|u\|_{X^\frac{1}{\mu_1}(\R^n)}^{\eta_2(1-\eta_1)}
        \|u\|_{X^\frac{1}{\mu_3}(\R^n)}^{(1-\eta_1)(1-\eta_2)}.
    \end{equation*}
    Simplify it and we obtain the interpolation inequality in Lemma.
\end{proof}

The next lemma addresses interpolation problems related to summation in Hölder norms.

\begin{lemma}\label{LemmaInterpolationConcerningSums}
    Assume $a,b,c,d>0$ and $0<\eta<1$. Then
    \begin{equation}\label{InterpolationConcerningSums}
        a^{\eta}b^{1-\eta}+c^{\eta}d^{1-\eta}
        \le(a+c)^{\eta}(b+d)^{1-\eta}.
    \end{equation}
\end{lemma}

\begin{proof}    
    The proof is simple. From Young's inequality we have
    \begin{align*}
        \l\frac{a}{a+c}\r^{\eta}\l\frac{b}{b+d}\r^{1-\eta}&
        \le\eta\cdot\frac{a}{a+c}+(1-\eta)\cdot\frac{b}{b+d}, \\
        \l\frac{c}{a+c}\r^{\eta}\l\frac{d}{b+d}\r^{1-\eta}&
        \le\eta\cdot\frac{c}{a+c}+(1-\eta)\cdot\frac{d}{b+d}.
    \end{align*}
    Combining these two inequalities and simplifying them, we conclude the proof.
\end{proof}

Now we can begin the proof of the interpolation lemma. According to Lemma \ref{LemmaTransitivityOfInterpolation}, we will focus on presenting the proof for all parameters being greater or less than zero, and the special case where the middle term $\frac{1}{\mu}$ is exactly equal to $0$, thereby naturally deducing the general case.

\begin{proof}[Proof of Lemma \ref{LemmaInterpolation}]
    \textbf{Case 1. $\lambda\ge0$.} In this case (\ref{interpolationInequality}) is merely the interpolation inequality for $L^p$, provable by applying Hölder's inequality. Set $\eta=\frac{\nu-\mu}{\nu-\lambda}$, then
    \begin{align*}
        \l\int_{\R^n}|u|^\frac{1}{\mu}\d x\r^\mu
        &=\l\int_{\R^n}|u|^\frac{\eta}{\mu}|u|^\frac{1-\eta}{\mu}\d x\r^\mu \\
        &\le\l\int_{\R^n}|u|^\frac{1}{\lambda}\d x\r^{\lambda\eta}
        \l\int_{\R^n}|u|^\frac{1}{\nu}\d x \r^{\nu(1-\eta)}.
    \end{align*}

    \textbf{Case 2. $\nu\le 0$.} In this case, all parameters are less than 0, so we only need to consider H\"older norms. For convenience, when $-\infty<\lambda<\mu<\nu\le 0$ and $k\in\N$, let us denote $C_\lambda$, $C_\mu$, $C_\nu$ and $C_k$ as
    \begin{equation*}
        C_\lambda=[u]_{C^{\lambda_1,\lambda_2}}, \quad
        C_\mu=[u]_{C^{\mu_1,\mu_2}}, \quad
        C_\nu=[u]_{C^{\nu_1,\nu_2}}, \quad
        C_k=\|u\|_{C^k}.
    \end{equation*}
    The parameters $\lambda_1,\lambda_1,\mu_1,\mu_2,\nu_1,\nu_2$ are defined as follows:
    \begin{align*}
        &\lambda_1=-\left[n\lambda+1\right], \quad \lambda_2=-n\lambda-\lambda_1, \\
        &\mu_1=-\left[n\mu+1\right], \quad \mu_2=-n\mu-\mu_1, \\
        &\nu_1=-\left[n\nu+1\right], \quad \nu_2=-n\nu-\nu_1.
    \end{align*}
    It is easy to check
    \begin{align*}
        \|u\|_{X^\frac{1}{\lambda}(\R^n)}&=C_{\lambda_1}+C_\lambda, \\
        \|u\|_{X^\frac{1}{\mu}(\R^n)}&=C_{\mu_1}+C_\mu, \\
        \|u\|_{X^\frac{1}{\nu}(\R^n)}&=C_{\nu_1}+C_\nu.
    \end{align*}
    Moreover, $\lambda<0$ ensure the continuity of function, when $\nu=0$ we have
    \begin{equation*}
        \|u\|_{L^\infty(\R^n)}=\|u\|_{C^0(\R^n)}=:\|u\|_{C^{0,0}(\R^n)}.
    \end{equation*}
    Based on the conditions, it is clear that
    \begin{equation*}
        \lambda_1\ge\mu_1\ge\nu_1.
    \end{equation*}    
    
    \noindent
    We divide the proof of this case into three sub-cases.

    \textbf{Case 2.1. Interpolation of $C_k$.} First assume $u\in C^2(\R^n)$, then from Taylor formula, we have
    \begin{equation*}
        u(y)-u(x)=Du(x)\cdot(y-x)+\frac{1}{2}(y-x)D^2 u\l(1-\theta)x+\theta y\r(y-x)^T.
    \end{equation*}
    For each direction $e_i(i=1,\cdots,n)$ and $h>0$,
    \begin{equation*}
        |u(x+he_i)-u(x)-D_i u(x)h|
        \le\frac{1}{2}\|D_{ii} u(x)\|_{L^\infty} h^2
        \le\frac{1}{2}\|D^2 u(x)\|_{L^\infty} h^2,
    \end{equation*}
    so that
    \begin{align*}
        \|D_i u(x)\|_{L^\infty}&\le\frac{2\|u(x)\|_{L^\infty}}{h}+\frac{1}{2}\|D^2 u\|_{L^\infty} h \\
        &\le 2(\|u(x)\|_{L^\infty})^\frac{1}{2}(\|D^2 u(x)\|_{L^\infty})^\frac{1}{2}.
    \end{align*}
    So we have
    \begin{align*}
        C_1&=\max(\|u\|_{L^\infty},\|Du\|_{L^\infty}) \\
        &\le(\|u\|_{L^\infty})^\frac{1}{2}
        \left[\max\l(\|u\|_{L^\infty}),2(\|D^2 u\|_{L^\infty})\r\right]^\frac{1}{2} \\
        &\lesssim C_0^\frac{1}{2} C_2^\frac{1}{2}.
    \end{align*}
    For $u\in C^{k+2}(\R^n)$, substituting $D^k u$ for $u$ directly yields
    \begin{equation*}
        C_{k+1}\lesssim C_k^\frac{1}{2} C_{k+2}^\frac{1}{2}.
    \end{equation*}
    Combine with Lemma \ref{LemmaTransitivityOfInterpolation} we know that for any $k_1,k_2,k_3\in\N$ and $k_1<k_2<k_3$, if $u\in C^{k_3}(\R^n)$,
    \begin{equation}\label{InterpolationOfCk}
        C_{k_2}\lesssim C_{k_1}^\eta C_{k_3}^{1-\eta}, 
    \end{equation}
    where
    \begin{equation*}
        \frac{1}{k_2}=\frac{\eta}{k_2}+\frac{1-\eta}{k_3}.
    \end{equation*}

    \textbf{Case 2.2. $\lambda_1-\nu_1=0$.} Without loss of generosity, we assume $\nu_1=0$. According to the condition $\lambda_1\ge\mu_1\ge\nu_1$ we have
    \begin{equation*}
        \lambda_1=\mu_1=\nu_1=0.
    \end{equation*}
    A direct calculation shows us
    \begin{equation}\label{ParameterCondition1}
        \mu_2=\eta\lambda_2+(1-\eta)\nu_2.
    \end{equation}
    Then apply the H\"older conditions and (\ref{ParameterCondition1}),
    \begin{align*}
        |u(x)-u(y)|
        &\le|u(x)-u(y)|^\eta
        |u(x)-u(y)|^{1-\eta} \\
        &\le C_\lambda^\eta C_\nu^{1-\eta}|x-y|^{\eta\lambda_2+(1-\eta)\nu_2} \\
        &=C_\lambda^\eta C_\nu^{1-\eta}|x-y|^{\mu_2}.
    \end{align*}
    This tells us
    \begin{equation*}
        C_\mu\le C_\lambda^\eta\cdot C_\nu^{1-\eta}.
    \end{equation*}
    Applying the interpolation inequality (\ref{InterpolationConcerningSums}) we have
    \begin{align*}
        \|u\|_{X^\frac{1}{\mu}(\R^n)}&=C_0+C_\mu \\
        &=C_0^\eta C_0^{1-\eta}+C_\lambda^\eta\cdot C_\nu^{1-\eta} \\
        &\le(C_0+C_\lambda)^{\eta}(C_0+C_\nu)^{1-\eta} \\
        &=\|u\|_{X^\frac{1}{\lambda}(\R^n)}^{\eta}\|u\|_{X^\frac{1}{\nu}(\R^n)}^{1-\eta}.
    \end{align*}

    \textbf{Case 2.3. $\lambda_1-\nu_1=1$ and $\mu_2=1$.} Without loss of generosity, we assume $\nu_1=0$. Since $\lambda_1\ge\mu_1\ge\nu_1$, then
    \begin{equation*}
        \lambda_1=\mu_1=\nu_1=0.
    \end{equation*}
    By a calculation, we obtain
    \begin{align*}
        1=\mu_2&=-n\mu-\mu_1 \\
        &=-n(\eta\lambda+(1-\eta)\nu)-\mu_1 \\
        &=\eta(1+\lambda_2)+(1-\eta)\nu_2,
    \end{align*}
    which equivalent to
    \begin{equation}\label{ParameterCondition2}
        \eta\lambda_2+(1-\eta)(\nu_2-1)=0.
    \end{equation}
    Since $u\in C_0^\infty(\R^n)$, there must exist point $x\in\R$ and direction $e_i$ such that $|D_i u(x)|=C_\mu$. Without loss of generality, we assume $|D_1 u(0)|=C_\mu$, then
    \begin{equation*}
        |D_1 u(0)-D_1 u(he_1)|\le|Du(0)-Du(he_1)|\le C_\lambda h^{\lambda_2}.
    \end{equation*}
    Thus we have
    \begin{equation*}
        |D_1 u(he_1)|\ge |D_1 u(0)|-|D_1 u(0)-D_1 u(he_1)|\ge C_\mu-C_\lambda h^{\lambda_2},
    \end{equation*}
    where
    \begin{equation*}
        h\in\left[0,\l\frac{C_\mu}{C_\lambda}\r^\frac{1}{\lambda_2}\right]:=[0,R].
    \end{equation*}
    It is worth noting when $h\in[0,R]$, $D_1 u(he_1)$ maintains the same symbol. Therefore
    \begin{align*}
        |u(0)-u(Re_1)|&=\int_0^R |D_1 u(he_1)|\d h \\
        &\ge\int_0^R\l C_\mu-C_\lambda h^{\lambda_2}\r\d h \\
        &=C_\mu R-\frac{C_\lambda}{\lambda_2+1}R^{\lambda_2+1}.
    \end{align*}
    Recall $u\in C^{\nu_1,\nu_2}$, thus we have the estimate
    \begin{equation*}
        |u(0)-u(Re_1)|\le C_\nu R^{\nu_2}.
    \end{equation*}
    Combing them together we obtain
    \begin{equation*}
        C_\mu R-\frac{C_\lambda}{\lambda_2+1}R^{\lambda_2+1}\le C_\nu R^{\nu_2}.
    \end{equation*}
    Substitute $R=\l\frac{C_\mu}{C_\lambda}\r^\frac{1}{\lambda_2}$ into the inequality,
    \begin{equation*}
        \frac{\lambda_2}{\lambda_2+1} C_\mu\le
        C_\nu\l\frac{C_\mu}{C_\lambda}\r^\frac{\nu_2-1}{\lambda_2}.
    \end{equation*}
    Recall the condition of parameters (\ref{ParameterCondition2}) we have
    \begin{equation*}
        \frac{\nu_2-1}{\lambda_2}=-\frac{\eta}{1-\eta}.
    \end{equation*}
    Thus we simplify and obtain
    \begin{equation}
        C_\mu\le\l1+\frac{1}{\lambda_2}\r^{1-\eta} C_\lambda^\eta C_\nu^{1-\eta},
    \end{equation}
    where $\lambda_2$ is independent of $u$.
    
    By applying the interpolation inequality for sums (\ref{InterpolationConcerningSums}), we derive the key interpolation inequality (\ref{interpolationInequality}). Combining these results with Lemma \ref{LemmaTransitivityOfInterpolation}, we establish the proof for the case $\nu \le 0$.

    \textbf{Case 3. $\mu=0$.} Finally, let's consider the case that spans both Lebesgue and H\"older norms. In this case we have $\lambda<0$ and $\nu>0$. First, let us assume $-\frac{1}{n}\le\lambda<0$ and $0<\nu\le 1$. Followed the notation we have
    \begin{align*}
        \lambda_1&=-\left[n\lambda+1\right]=0, \\
        \lambda_2&=-n\lambda-\lambda_1=-n\lambda.
    \end{align*}
    Define
    \begin{equation*}
        a+b:=\|u\|_{L^\infty(\R^n)}+\sup_{x,y\in\R^n}\frac{|u(x)-u(y)|}{|x-y|^{\lambda_2}}
        =\|u\|_{X^\frac{1}{\lambda}(\R^n)}<\infty
    \end{equation*}
    for short. Easy to see $a<\infty$ and $u$ is continuous. When $b=0$ the function is trivial. Without loss of generality, assume $|u(0)|=a$, and we obtain
    \begin{equation*}
        |a-u(x)|\le b|x|^{\lambda_2}.
    \end{equation*}
    Therefore, we have
    \begin{equation*}
        |u(x)|\ge a-|a-u(x)|\ge a-b|x|^{\lambda_2}, \quad 
        |x|\in\left[0,\l\frac{a}{b}\r^\frac{1}{\lambda_2}\right]:=[0,c].
    \end{equation*}
    Set $p:=\frac{1}{\nu}\ge 1$ yields
    \begin{align*}
        \int_{\R^n}|u|^p\d x&\ge\int_{B(0,c)}|u|^p\d x \\
        &=\int_0^c\int_{\p B(0,r)}|u|^p\d S\d r \\
        &\gtrsim\int_0^c(a-br^{\lambda_2})^p r^{n-1}\d r \\
        &=\frac{a^p c^n}{\lambda_2}\int_0^1(1-s)^p s^{\frac{n}{\lambda_2}-1}\d s
        & &(r=cs^\frac{1}{\lambda_2}) \\
        &=\frac{a^p c^n}{\lambda_2}\B\l-\frac{1}{\lambda},p-1\r.
    \end{align*}
    Notice that
    \begin{equation*}
        (a^p c^n)^\frac{1}{p}=a^{1-\frac{\nu}{\lambda}} b^{\frac{\nu}{\lambda}}
        =a^\frac{1}{1-\eta}b^\frac{-\eta}{1-\eta}.
    \end{equation*}
    Finally we conclude the estimate
    \begin{align*}
        \|u\|_{X^\frac{1}{\mu}(\R^n)}&=\|u\|_{L^\infty(\R^n)} \\
        &=a \\
        &\le\l\frac{a}{b}+1\r^\eta a \\
        &=(a+b)^\eta(a^\frac{1}{1-\eta}b^\frac{-\eta}{1-\eta})^{1-\eta} \\
        &\lesssim\|u\|_{X^\frac{1}{\lambda}(\R^n)}^{\eta}
        \|u\|_{X^\frac{1}{\nu}(\R^n)}^{1-\eta}.
    \end{align*}
    By repeatedly applying the transitivity property of the interpolation inequality, as outlined in Lemma \ref{LemmaTransitivityOfInterpolation}, we systematically extend our results to cover all parameter ranges. This approach allows us to derive the remaining cases, thereby completing the proof of the interpolation theorem for the full spectrum of parameters.
\end{proof}

\section{Proof of Theorem \ref{TheoremGagliardoNirenbergInequalityGeneral}}

In this section, we will prove our conclusions using mathematical induction and the Interpolation Lemma \ref{LemmaInterpolation} we established in section 3. The basic approach is similar to the method used by L. Nirenberg in \cite{nirenberg2011elliptic}, with additional reference to the detailed proofs provided by A. Fiorenza, M.R. Formica, T. Roskovec, and F. Soudsk\'y in \cite{fiorenza2021detailed}.

\subsection{Initial Inequality}

To prove the Gagliardo-Nirenberg inequality by mathematical induction, we need a fundamental base inequality. Let’s examine the case when $l=1$ and $k=2$ in Theorem \ref{TheoremGagliardoNirenbergInequalityGeneral}.

\begin{lemma}\label{LemmaGagliardoNirenbergInequalityInitial}
    Let $-\infty<\frac{1}{p},\frac{1}{q},\frac{1}{r}\le 1$ and
    \begin{equation*}
        \frac{2}{q} = \frac{1}{p} + \frac{1}{r}.
    \end{equation*}
    Then for all $u\in C_0^\infty(\R^n)$ with $u\in X^{2,p}(\R^n)\cap X^r(\R^n)$, it follows that $u\in X^{1,q}(\R^n)$ and there exists a constant $C$ independent of $u$ such that
    \begin{equation}\label{GagliardoNirenbergInequalityInitial}
        \|Du\|_{X^q(\R^n)}^2 \le C \|D^2 u\|_{X^p(\R^n)} \|u\|_{X^r(\R^n)},
    \end{equation}
\end{lemma}

\begin{remark}
    Let us first discuss the different parameter cases involved in this lemma, some of which have already been proven.
    \begin{enumerate}
        \item When $p,q,r\ge 1$, the inequality involves only Lebesgue norms. This case was left as an exercise in L. Nirenberg's paper \cite{nirenberg2011elliptic}. Recently, A. Fiorenza, M.R. Formica, T. Roskovec and F. Soudsk\'y provided a detailed proof of this lemma under these conditions in their paper \cite{fiorenza2021detailed}.
        \item For $r\in[-\infty,n)$ and $p,q\in[1,\infty)$, relevant results can be found in the paper by A. Kufner and A. Wannebo \cite{kufner1995interpolation}.
        \item For the remaining cases, which currently lack established proofs, we will categorize them into two types for detailed analysis: $r\in[-n,0)$ and $q\in[-\infty,0)$. Since the value of $p$ is constrained by the parameter condition $\frac{2}{q} = \frac{1}{p} + \frac{1}{r}$, it does not require separate consideration. This division allows us to focus on the critical ranges of $r$ and $q$ while ensuring a systematic and comprehensive treatment of all cases.
    \end{enumerate}
\end{remark}

\begin{proof}[Proof of Lemma \ref{LemmaGagliardoNirenbergInequalityInitial}]
    We denote the converse Sobolev conjugate index $p_\ast$ with
    \begin{equation*}
        \frac{1}{p_\ast}=\frac{1}{p}+\frac{1}{n}.
    \end{equation*}    
    
    Consider the first case when $r\in[-n,0)$. Easy to verify
    \begin{equation*}
        \frac{2}{q}=\frac{1}{p^\ast}+\frac{1}{r_\ast}.
    \end{equation*}
    Note that
    \begin{equation*}
        \frac{1}{r_\ast}=\frac{1}{r}+\frac{1}{n}<0,
    \end{equation*}
    therefore, we have $r_\ast\in[-\infty,0)$, then apply the general Sobolev inequality (\ref{SobolevInequalityGeneral}), H\"older equality (\ref{HolderEquality}) and the interpolation inequality (\ref{interpolationInequality}), we have
    \begin{align*}
        \|Du\|_{X^{p^\ast}(\R^n)}&\lesssim\|D^2 u\|_{X^p(\R^n)}, \\
        \|Du\|_{X^{r_\ast}(\R^n)}&=\|u\|_{X^r(\R^n)}, \\
        \|Du\|_{X^q(\R^n)}^2&\lesssim \|Du\|_{X^{p^\ast}(\R^n)}\|Du\|_{X^{r_\ast}(\R^n)}.
    \end{align*}
    Combining them, we obtain inequality (\ref{GagliardoNirenbergInequalityInitial})
    \begin{equation*}
        \|Du\|_{X^q(\R^n)}^2 \le C \|D^2 u\|_{X^p(\R^n)} \|u\|_{X^r(\R^n)}.
    \end{equation*}

    For the second case when $q\in[-\infty,0)$. The parameters satisfy the condition
    \begin{equation*}
        \frac{2}{q^\ast}=\frac{1}{p^{\ast\ast}}+\frac{1}{r}.
    \end{equation*}
    Use (\ref{SobolevInequalityGeneral}), (\ref{HolderEquality}) and (\ref{interpolationInequality}) again we have
    \begin{align*}
        \|u\|_{X^{p^{\ast\ast}}(\R^n)}&\lesssim\|D^2 u\|_{X^p(\R^n)}, \\
        \|Du\|_{X^q(\R^n)}&=\|u\|_{X^{q^{\ast}}(\R^n)}, \\
        \|u\|_{X^{q^\ast}(\R^n)}^2&\lesssim \|u\|_{X^{p^{\ast\ast}}(\R^n)}\|u\|_{X^r(\R^n)}.
    \end{align*}
    By combining these, we arrive once again at inequality (\ref{GagliardoNirenbergInequalityInitial})
    \begin{equation*}
        \|Du\|_{X^q(\R^n)}^2 \le C \|D^2 u\|_{X^p(\R^n)} \|u\|_{X^r(\R^n)}.
    \end{equation*}

    With the discussion of these two cases, we complete the proof of the lemma.
\end{proof}

\subsection{Induction}

When $\theta=1$, Sobolev inequality (\ref{SobolevInequalityGeneral}) arises as a special case of the Gagliardo-Nirenberg inequality (\ref{GagliardoNirenbergInequalityGeneral}). In this subsection, we focus on the case $\theta=\frac{l}{k}$, where the parameters hold the condition:
\begin{equation*}
    \frac{1}{q} = \frac{l}{k}\cdot\frac{1}{p}+\frac{k-l}{k}\cdot\frac{1}{r}.
\end{equation*}

We prove the inequality for a pair of integers $(l,k)$ by mathematical induction. For the base case, let's consider $l=1$ and $k=2$, which has been addressed in Lemma \ref{LemmaGagliardoNirenbergInequalityInitial} in the preceding subsection.

First, we use induction method on the parameter $k$, assume the Gagliardo-Nirenberg inequality (\ref{GagliardoNirenbergInequalityGeneral}) holds for $l=1$ and $k=\Tilde{k}\in\N$. We want to show it also holds for $l=1$ and $k=\Tilde{k}+1$. Notice the parameters satisfy
\begin{equation}\label{parameterRelation1}
    \frac{1}{q}=\frac{1}{\Tilde{k}+1}\cdot\frac{1}{p}+\frac{\Tilde{k}}{\Tilde{k}+1}\cdot\frac{1}{r}.
\end{equation}
Set $s$ corresponding to $q,r$ such that
\begin{equation}\label{parameterRelation2}
    \frac{2}{q}=\frac{1}{s}+\frac{1}{r}.
\end{equation}
Then comes from (\ref{GagliardoNirenbergInequalityInitial}) we obtain
\begin{equation}\label{inequalityQSR}
    \|D u\|_{X^q(\R^n)} \le C \|D^2 u\|_{X^s(\R^n)}^\frac{1}{2} \|u\|_{X^r(\R^n)}^\frac{1}{2}.
\end{equation}
From the relations (\ref{parameterRelation1}) and (\ref{parameterRelation2}) of the parameters we derive that
\begin{equation*}
    \frac{1}{s}=\frac{1}{\Tilde{k}}\cdot\frac{1}{p}+\frac{\Tilde{k}-1}{\Tilde{k}}\cdot\frac{1}{q},
\end{equation*}
which coincides with the case $l=1$ and $k=\Tilde{k}$ we assumed, therefore apply (\ref{GagliardoNirenbergInequalityGeneral}) on $Du$ we have
\begin{equation}\label{inequalitySPQ}
    \|D^2 u\|_{X^s(\R^n)} \le C \|D^{\Tilde{k}+1} u\|_{X^p(\R^n)}^\frac{1}{\Tilde{k}} \|D u\|_{X^q(\R^n)}^\frac{\Tilde{k}-1}{\Tilde{k}}.
\end{equation}
Then inequalities (\ref{inequalityQSR}) and (\ref{inequalitySPQ}) imply that
\begin{equation*}
    \|D u\|_{X^q(\R^n)} \le C \|D^{\Tilde{k}+1} u\|_{X^p(\R^n)}^\frac{1}{\Tilde{k}+1} \|u\|_{X^r(\R^n)}^\frac{\Tilde{k}}{\Tilde{k}+1}.
\end{equation*}
Which completes the first induction.

Secondly, we use induction concerning the parameter $j$, assuming the Gagliardo-Nirenberg inequality (\ref{GagliardoNirenbergInequalityGeneral}) holds for $(l,k)=(\Tilde{l},\Tilde{k})\in\mathbb{N}^2$ with $\Tilde{l}<\Tilde{k}$. It is enough to prove it also holds for $(l,k)=(\Tilde{l}+1,\Tilde{k}+1)$. Similarly, we have the parameters satisfy
\begin{equation}\label{parameterRelation3}
    \frac{1}{q}=\frac{\Tilde{l}+1}{\Tilde{k}+1}\cdot\frac{1}{p}
    +\frac{\Tilde{k}-\Tilde{l}}{\Tilde{k}+1}\cdot\frac{1}{r}.
\end{equation}
Set $t$ to be such that
\begin{equation}\label{parameterRelation4}
    \frac{1}{q}=\frac{\Tilde{l}}{\Tilde{k}}\cdot\frac{1}{p}
    +\frac{\Tilde{k}-\Tilde{l}}{\Tilde{k}}\cdot\frac{1}{t}.
\end{equation}
Thus the parameters meet the assumption, apply (\ref{GagliardoNirenbergInequalityGeneral}) on $Du$ we get
\begin{equation}\label{inequalityQPT}
    \|D^{\Tilde{l}+1} u\|_{X^q(\R^n)} \le C \|D^{\Tilde{k}+1} u\|_{X^p(\R^n)}^\frac{\Tilde{l}}{\Tilde{k}}
    \|D u\|_{X^t(\R^n)}^\frac{\Tilde{k}-\Tilde{l}}{\Tilde{k}}.
\end{equation}
Then (\ref{parameterRelation3}) and (\ref{parameterRelation4}) imply that
\begin{equation*}
    \frac{1}{t}=\frac{1}{\Tilde{l}+1}\cdot\frac{1}{q}
    +\frac{\Tilde{l}}{\Tilde{l}+1}\cdot\frac{1}{r}.
\end{equation*}
Let $l=1$ and $k=\Tilde{l}+1$, it matched the result we have for the first induction
\begin{equation}\label{inequalityTQR}
    \|D u\|_{X^t(\R^n)} \le C \|D^{\Tilde{l}+1} u\|_{X^q(\R^n)}^\frac{1}{\Tilde{l}+1} \|u\|_{X^r(\R^n)}^\frac{\Tilde{l}}{\Tilde{l}+1}.
\end{equation}
Combining (\ref{inequalityQPT}) and (\ref{inequalityTQR}), we complete the second induction,
\begin{equation*}
    \|D^{\Tilde{l}+1} u\|_{X^q(\R^n)} \le C \|D^{\Tilde{k}+1} u\|_{X^p(\R^n)}^\frac{\Tilde{l}+1}{\Tilde{k}+1} \|u\|_{X^r(\R^n)}^\frac{\Tilde{k}-\Tilde{l}}{\Tilde{k}+1}.
\end{equation*}

\subsection{Interpolation}

Once we reach the inequalities for endpoints of $\theta=1$ and $\theta=\frac{l}{k}$, the general case is followed by the interpolation Lemma \ref{LemmaInterpolation}. Set
\begin{gather*}
    \frac{1}{q_1}-\frac{l}{n}=\frac{1}{p}-\frac{k}{n}, \\
    \frac{1}{q_2} = \frac{l}{k}\cdot\frac{1}{p}+\frac{k-l}{k}\cdot\frac{1}{r}, \\
    \eta+\frac{l}{k}\cdot(1-\eta)=\theta.
\end{gather*}
It is easy to verify that
\begin{equation*}
    \frac{1}{q}=\frac{\eta}{q_1}+\frac{1-\eta}{q_2}.
\end{equation*}
Then
\begin{align*}
    \|D^l u\|_{X^q(\R^n)}&\lesssim \|D^l u\|_{X^{q_1}(\R^n)}^\eta \|D^l u\|_{X^{q_2}(\R^n)}^{1-\eta} \\
    &\lesssim \|D^k u\|_{X^p(\R^n)}^\eta \|D^k u\|_{X^p(\R^n)}^{\frac{l}{k}(1-\eta)} \|u\|_{X^r(\R^n)}^{\frac{k-l}{k}(1-\eta)} \\
    &=\|D^k u\|_{X^p(\R^n)}^\theta \|u\|_{X^r(\R^n)}^{1-\theta}.
\end{align*}
By synthesizing the results derived from the initial inequality, the inductive argument, and the interpolation techniques, we successfully complete the proof of Theorem \ref{TheoremGagliardoNirenbergInequalityGeneral}.









\begin{thebibliography}{99}  

\bibitem{brezis2011functional} Brezis, Haim, and Haim Brézis. Functional analysis, Sobolev spaces and partial differential equations. Vol. 2. No. 3. New York: Springer, 2011.

\bibitem{calderon1964intermediate} Calderón, Alberto. “Intermediate spaces and interpolation, the complex method." Studia Mathematica 24.2 (1964): 113-190.

\bibitem{campanato1963proprieta} Campanato, Sergio. “Proprietà di hölderianità di alcune classi di funzioni." Annali della Scuola Normale Superiore di Pisa-Scienze Fisiche e Matematiche 17.1-2 (1963): 175-188.

\bibitem{evans2010partial} Evans, Lawrence C. Partial differential equations. Vol. 19. American Mathematical Society, 2022.

\bibitem{fiorenza2021detailed} Fiorenza, Alberto, et al. “Detailed proof of classical Gagliardo–Nirenberg interpolation inequality with historical remarks." Zeitschrift für Analysis und ihre Anwendungen 40.2 (2021): 217-236.

\bibitem{gagliardo1959ulteriori} Gagliardo, Emilio. “Ulteriori proprietà di alcune classi di funzioni in più variabili." Ricerche Mat. 8 (1959): 24-51.

\bibitem{gilbarg1977elliptic} Gilbarg, David, and Neil S. Trudinger. Elliptic partial differential equations of second order. Vol. 224. No. 2. Berlin: springer, 1977.

\bibitem{john1961functions} John, Fritz, and Louis Nirenberg. “On functions of bounded mean oscillation." Communications on pure and applied Mathematics 14.3 (1961): 415-426.

\bibitem{kufner1995interpolation} Kufner, Alois, and Andreas Wannebo. “An interpolation inequality involving Hölder norms." (1995): 603-612.

\bibitem{lunardi2012analytic} Lunardi, Alessandra. Analytic semigroups and optimal regularity in parabolic problems. Springer Science \& Business Media, 2012.

\bibitem{lunardi2018interpolation} Lunardi, Alessandra. Interpolation theory. Vol. 16. Springer, 2018.

\bibitem{meyers1964mean} Meyers, Norman G. “Mean oscillation over cubes and Hölder continuity." Proceedings of the American Mathematical Society 15.5 (1964): 717-721.

\bibitem{nirenberg2011elliptic} Nirenberg, Louis. “On elliptic partial differential equations." Annali della Scuola Normale Superiore di Pisa-Scienze Fisiche e Matematiche 13.2 (1959): 115-162.

\bibitem{stampacchia1965spaces} Stampacchia, Guido. “The spaces $\mathcal {L}^{(p,\lambda)}, N^{(p,\lambda)} $ and interpolation." Annali della Scuola Normale Superiore di Pisa-Scienze Fisiche e Matematiche 19.3 (1965): 443-462.

\bibitem{soudsky2018interpolation} Soudský, Filip, Anastasia Molchanova, and Tomáš Roskovec. “Interpolation between Hölder and Lebesgue spaces with applications." Journal of Mathematical Analysis and Applications 466.1 (2018): 160-168.

\bibitem{triebel1995interpolation} Triebel, Hans. "Interpolation theory, function spaces, differential operators." JA Barth (1995).




\end{thebibliography}

\end{document}